\documentclass{lms}
\usepackage{amsmath,systeme,amsopn,amssymb,mathrsfs,times,newtxtext,newtxmath,upref}
\usepackage[mathscr]{eucal}
\usepackage{color}
\usepackage{enumitem}
\setlist{label={$$\roman{enumi}\kern1pt$)$}}

\newtheorem{thm}{Theorem}[section]
\newtheorem{prop}[thm]{Proposition}
\newtheorem{cor}[thm]{Corollary}

\numberwithin{equation}{section}

\newenvironment{demode}
  {\noindent {{\it Proof of }}}%
  {\par \hfill \fbox{}}

\newcommand{\St}{\mathcal{S}}

\newcommand{\HH}{\mathcal{H}}

\newcommand{\M}{\mathcal{M}}

\newcommand{\N}{\mathcal{N}}
\newcommand{\X}{\mathcal{X}}

\newcommand{\mc}[1]{\mathcal{#1}}

\newcommand{\ra}{\rightarrow}

\DeclareMathOperator{\lat}{Lat}

\DeclareMathOperator{\Span}{span}

\title[A note on a Halmos problem]{A note on a Halmos problem}

\author{Maximiliano Contino and Eva A. Gallardo-Guti\'errez}

\classno{47A15 (primary), 47B37, 47B38 (secondary)}

\extraline{First author is supported by María Zambrano Postdoctoral Grant CT33/21 at Universidad Complutense de Madrid financed by the Ministry of Universities with Next Generation EU funds, by Grant CEX2019-000904-S funded by MCIN/AEI/10.13039/501100011033 and by CONICET PIP 11220200102127CO. Second author is partially supported by Plan Nacional  I+D grant no. PID2019-105979GB-I00, Spain, the Spanish Ministry of Science and Innovation, through the ``Severo Ochoa Programme for Centres of Excellence in R\&D'' (CEX2019-000904-S) and from the Spanish National Research Council, through the ``Ayuda extraordinaria a Centros de Excelencia Severo Ochoa'' (20205CEX001). }

\date{July 2023, revised July 2024}

\begin{document}

\maketitle

\begin{abstract}
		We address the existence of non-trivial closed invariant subspaces of operators $T$ on Banach spaces whenever their square $T^2$ have or, more generally, whether there exists a polynomial $p$ with $\mbox{deg}(p)\geq 2$ such that the lattice of invariant subspaces of $p(T)$ is non-trivial. In the Hilbert space setting, the $T^2$-problem was posed by Halmos in the seventies and in 2007, Foias, Jung, Ko and Pearcy \cite{Foias} conjectured it could be equivalent to the \emph{Invariant Subspace Problem}.
\end{abstract}

\section{Introduction and Preliminaries}

In 1986 Charles Read \cite{Re1986} showed the existence of a linear bounded operator $R$ on the Banach space $\ell^1$ such that while $R$ has no non-trivial closed invariant subspaces, all its powers $R^n$ for $n\geq 2$ do (see also \cite{Re1985}). More recently, Gallardo-Guti\'errez and Read \cite{GR} constructed a \emph{quasinilpotent operator}  such that no non-constant polynomial $p(T)$ has non-trivial closed invariant subspaces (even no non-constant analytic germ $f(T)$).

\smallskip

Broadly speaking, Read's construction provided a counterexample in the setting of Banach spaces to a problem raised by Halmos in the seventies  in the context of Hilbert spaces in \cite{Halmos} namely, whether every linear bounded operator $T$ on a Hilbert space such that its square $T^2$ has a non-trivial closed invariant subspace has necessarily a non-trivial closed invariant subspace. Apparently, although the latter problem remains open \footnote{The article in which Per Enflo claims to have solved the Invariant Subspace Problem in the setting of Hilbert spaces (see \texttt{https://arxiv.org/abs/2305.15442}) is being reviewed for publication at the time of writing.}, there have been considerable efforts to solve it (we refer to the work of Foias, Jung, Ko and Pearcy \cite{Foias} and the references therein).

\smallskip

In this framework, our goal is to study linear bounded operators $T$ acting on complex Banach spaces $\mc X$ such that $T^2$, or more generally a polynomial in $T$, has non-trivial closed invariant subspaces. Note that $\mc X$ must be an infinite dimensional separable Banach space and the polynomial must be of degree strictly greater than $1$, otherwise the problem is straightforward. In this setting, our main result reads as follows:

\begin{thm*} Let $\X$ be an infinite dimensional separable complex Banach space, $T$ a quasinilpotent linear bounded operator on $\X$ and $n\geq 1$. The following conditions are equivalent:
	\begin{enumerate}
        \item $T$ has a nontrivial closed invariant subspace;
		\item $(T-\lambda I)^n$ has a nontrivial closed invariant subspace for every $\lambda \in \mathbb{C};$
		\item $(T-\lambda I)^n$ has a nontrivial closed invariant subspace for some $\lambda \in \mathbb{C} \setminus \{0\}.$
	\end{enumerate}
\end{thm*}

The proof of the stated theorem as well as related results addressing the existence of non-trivial closed invariant subspaces of quasinilpotent operators $T$ on Banach spaces whenever there exists a polynomial $p$ with $\mbox{deg}(p)\geq 2$ such that the lattice of invariant subspaces of $p(T)$ is non-trivial will be shown in Section \ref{Section 2}.

\smallskip

In Section \ref{Section 3}, we deal with a matrix approach of the problem considered in \cite{Foias} allowing us to show conditions which ensure that the lattices of $T$ and $T^2$ coincide. Moreover, if $T\in L(\X)$, the Banach algebra of the bounded linear operators on $\mc X$, and $A_T \in L(\X \times \X)$ denotes the operator given by
$$
A_T=\begin{bmatrix}
0 & I \\ T & 0  \end{bmatrix},
$$
Theorem \ref{teo1} states that $T$ has a nontrivial closed invariant subspace if and only if $A_T$ and $A_{-T}$ has a common nontrivial closed invariant subspace. As a consequence, we observe that if $A_T^2$ has a nontrivial closed hyperinvariant subspace, then $T$ has a nontrivial closed invariant subspace.

\smallskip

Throughout the text, given $T \in L(\X)$, the spectrum of $T$ will be noted by $\sigma(T)$, $\rho(T)$ will stand for  the resolvent of $T$, that is, $\rho(T)=\mathbb{C} \setminus \sigma(T)$, and the full spectrum, namely, the union of $\sigma(T)$ and all the bounded connected components of $\rho(T)$, will be denoted by $\eta(\sigma(T))$. Finally, recall that $T$ is quasinilpotent if $\sigma(T)=\{0\}$.

\section{Polynomial equations and invariant subspaces} \label{Section 2}

In this section, we deal with the existence of non-trivial closed invariant subspaces for $T$ assuming particular polynomials in $T$ do have. As stated in the introduction, the following theorem holds:

\begin{thm} \label{thm1} Let $\X$ be an infinite dimensional separable complex Banach space, $T$ a quasinilpotent linear bounded operator on $\X$ and $n\geq 1$. The following conditions are equivalent:
	\begin{enumerate}
        \item $T$ has a nontrivial closed invariant subspace;
		\item $(T-\lambda I)^n$ has a nontrivial closed invariant subspace for every $\lambda \in \mathbb{C};$
		\item $(T-\lambda I)^n$ has a nontrivial closed invariant subspace for some $\lambda \in \mathbb{C} \setminus \{0\}.$
	\end{enumerate}
\end{thm}

In order to prove it, we will make use of the following result proved by Matache.

\begin{prop}[{\cite[Corollary 3]{Matache}}] \label{Matache} Suppose that $A, B \in L(\X)$ are commuting operators such that $A^n=B^n$ for some $n\geq 1.$ Then either both lattices of $A$ and $B$ are trivial or $A$ and $B$ have common nontrivial closed invariant subspaces.
\end{prop}

\begin{demode}{\it Theorem \ref{thm1}.} Let $n>1$ be fixed, otherwise there is nothing to prove. The only non-trivial implication is that condition $iii)$ implies that $T$ has non-trivial closed invariant subspaces. Hence, let us assume that $(T-\lambda_0 I)^n$ has a nontrivial closed invariant subspace for some $\lambda_0 \in \mathbb{C} \setminus \{0\}.$

It holds that $\sigma((T-\lambda_0I)^n)=\{{(-\lambda_0})^n\},$ because $T$ is quasinilpotent. Since $\lambda_0 \not = 0,$ there exists a connected open set $\mathcal{V}$ containing the full spectrum of $(T-\lambda_0I)^n$, such that
$$f(z):=z^{\frac{1}{n}}, \quad z\in \mathcal{V}$$ is analytic and one-to-one in $\mathcal{V}.$
Set $U:=f((T-\lambda_0I)^n),$ then $U$ is well defined, bounded and since $T-\lambda_0 I$ commutes with $(T-\lambda_0 I)^n$,  $U$ commutes with $T-\lambda_0 I.$

Theorem 2.14 in \cite{Radjavi} yields that the lattice of $(T-\lambda_0I)^n$ and $U$ coincides. Finally, since  $$U^n=(T-\lambda_0 I)^n,$$ by Proposition \ref{Matache}, $T-\lambda_0 I$ has a nontrivial closed invariant subspace. From here, condition $i)$ follows and Theorem \ref{thm1} is proved.
\end{demode}

\begin{cor} \label{corpoly} Let $\X$ be an infinite dimensional separable complex Banach space, $T\in L(\X)$ quasinilpotent and $p$ a polynomial of degree $n=\deg (p)\geq 2$ such that $p^\prime$ (the derivative of $p$) has a non-zero root of multiplicity $n-1.$ Then $T$  has a nontrivial closed invariant subspace if and only if $p(T)$ has a nontrivial closed invariant subspace.
\end{cor}

\begin{proof} By hypothesis there exist $\alpha, \lambda \in \mathbb{C}\setminus \{0\}$ and $\mu \in \mathbb{C}$ such that $p(z)=\alpha(z-\lambda)^n+\mu.$ Then $p(T)$ has a nontrivial closed invariant subspace if and only if $(T-\lambda I)^n$ has a nontrivial closed invariant subspace if and only if  $T$ has a nontrivial closed invariant subspace, where we used Theorem \ref{thm1}.
\end{proof}

Note that Corollary \ref{corpoly} does not imply anything about the possible multiplicity of the roots of $p.$ For instance, suppose that $p$ is a polynomial of degree $3.$ Then $p(z)=az^3+bz^2+cz+d$ with $a \in \mathbb{C} \setminus \{0\}$ and $b,c,d \in \mathbb{C}.$ It can be checked that $p'$ has a non-zero root of multiplicity $2$ if and only if $b\not =0$ and $b^2=3ac.$
It is clear that if $p$ has a non-zero root of multiplicity $3$ then $p'$ has the same root with multiplicity $2$ and we can apply Corollary \ref{corpoly}.

Nevertheless,  if we denote by $\alpha_i \in \mathbb{C}$ for $i=1,2,3$ the three roots of $p$, namely
$$p(z)=\delta(z-\alpha_1)(z-\alpha_2)(z-\alpha_3)$$
where $\delta \in \mathbb{C} \setminus \{0\}$, a straightforward computation yields that  $p'$ has a non-zero root of multiplicity $2$ if and only if $\alpha_1+\alpha_2+\alpha_3\not =0$ and $\delta^2(\alpha_1+\alpha_2+\alpha_3)^2=3\delta^2(\alpha_1 \alpha_2+\alpha_1 \alpha_3 + \alpha_2 \alpha_3),$ or equivalently $\alpha_1+\alpha_2+\alpha_3\not =0$ and
\begin{equation} \label{qua}
\alpha_1^2+\alpha_2^2+\alpha_3^2=\alpha_1 \alpha_2+\alpha_1 \alpha_3 + \alpha_2 \alpha_3.
\end{equation}
Clearly, in both cases Corollary \ref{corpoly} applies. But, if $\alpha_1=\alpha_2$ then,  by
\eqref{qua}, $2\alpha_1^2+\alpha_3^2=\alpha_1^2+2\alpha_1 \alpha_3$  and hence, $\alpha_1=\alpha_2=\alpha_3$ so Corollary \ref{corpoly} does not apply. More precisely, in order to apply Corollary \ref{corpoly}, $p$ must have either three (different) roots satisfying \eqref{qua} or a root of multiplicity three:

\begin{cor} Let $\X$ be an infinite dimensional separable complex Banach space, $T\in L(\X)$ quasinilpotent and $p(z)=az^3+bz^2+cz+d$ with $b \in \mathbb{C} \setminus \{0\}$ and $b^2=3ac$. Then $T$  has a nontrivial closed invariant subspace if and only if $p(T)$ has a nontrivial closed invariant subspace.
\end{cor}

\smallskip

For polynomials of degree $2$, the situation turns out to be more general:

\begin{thm} \label{propsq1} Let $\X$ be an infinite dimensional separable complex Banach space and $T\in L(\X)$ quasinilpotent. Then $T$  has a nontrivial closed invariant subspace if and only if $p(T)$ has a nontrivial closed invariant subspace for some (and then every) $p(z)=az^2+bz+c$ with $b \in \mathbb{C} \setminus \{0\}.$ Moreover, $T$  has a nontrivial closed invariant subspace if and only if $(p(T))^n$ has a nontrivial closed invariant subspace for some (and then every) $n\geq 1$ and $p(z)=az^2+bz+c$ with $b \in \mathbb{C} \setminus \{0\}.$
\end{thm}

\begin{proof} If $a=0$ then the result follows from Theorem \ref{thm1}. Suppose that $a \not =0$ and  complete squares to get $q(z):=\frac{p(z)}{a}=(z+\frac{b}{2a})^2+(\frac{c}{a}-\frac{b^2}{4a^2}).$ Clearly, for every $n\geq 1$, $q(T)^n$ has a nontrivial closed invariant subspace if and only if $p(T)^n$ has a nontrivial closed invariant subspace.
If $\frac{c}{a}= \frac{b^2}{4a^2}$ then $q(T)^n=(T+\frac{b}{2a}I)^{2n}$ and the conclusion follows again by Theorem \ref{thm1}. Finally, if $\frac{c}{a}\not = \frac{b^2}{4a^2},$ take $\lambda:=\frac{c}{a}-\frac{b^2}{4a^2} \not =0.$ Then, by Theorem \ref{thm1}, $q(T)^n=((T+\frac{b}{2a}I)^2+\lambda I)^n$ has a nontrivial closed invariant subspace if and only if $(T+\frac{b}{2a}I)^2$ has a nontrivial closed invariant subspace, or equivalently  $T$ has a nontrivial closed invariant subspace, where we used again Theorem \ref{thm1}. The converse always holds.
\end{proof}

\begin{cor} Let $T \in L(\X)$ quasinilpotent and $\lambda, \mu \in \mathbb{C}$ such that $\lambda \not = -\mu.$ Then $T$  has a nontrivial closed invariant subspace if and only if for some (and then every) $n\geq 1,$ $(T-\lambda I)^n(T-\mu I)^n$ has a nontrivial closed invariant subspace.
\end{cor}

\begin{proof} Note that $(T-\lambda I)^n(T-\mu I)^n=[(T-\lambda I)(T-\mu I)]^n=(T^2-(\lambda+\mu)T+ \lambda\mu I)^n$ with $\lambda+\mu \not =0.$ Then the result follows from Theorem \ref{propsq1}.
\end{proof}

\bigskip

Assuming further properties on the spectrum of the operator, one may deduce that for $k\geq 2$ the lattices of $T$ and $T^k$ coincide:

\begin{prop} \label{propcyclicn} Let $T \in L(\X)$ be such that $0$ is in the unbounded component of $\rho(T^k)$ for some $k \geq 2.$ Then $T$ is cyclic if and only if $T^k$ is cyclic. Moreover, the set of cyclic vectors coincide. In particular, the lattices of $T$ and $T^k$ coincide.
\end{prop}

\begin{proof} Arguing as in \cite[Proposition 4.2]{Ansari}, there exists a curve $\gamma$ in $\rho(T^k)$ joining $0$ to $\infty.$ Since $f(z)=z^{1/k}$ is analytic on the simply connected open set $\Omega := \mathbb{C} \setminus \Gamma,$ Runge’s Theorem yields a sequence $(p_n)_{n \geq 1}$ of polynomials such that $(p_n)_{n \geq 1}$ converges to $f$ uniformly on compact subsets of $\Omega.$ Since $\sigma(T^k) \subseteq \Omega,$ the Riesz functional calculus shows that $p_n(T^k)$ converges in the operator norm to $T.$
	
Now let $x$ be a cyclic vector for $T.$ Then for every $j \geq 0,$
$$\underset{n \ra \infty}{\lim}(p_n(T^k))^jx = T^jx.$$ It follows that $x$ must also be cyclic for $T^k.$ The converse always holds. The last statement is a straightforward consequence.
\end{proof}

\section{A matrix approach} \label{Section 3}
As mentioned in the Introduction, in 1970  Halmos \cite{Halmos} set ten open problems concerning operators in the setting of Hilbert spaces. In the discussion, a natural question arose:  if $T^2$ has a nontrivial closed invariant subspace, must $T$ have a nontrivial closed invariant subspace also? Recently, in \cite{Foias}, the authors addressed a matrix approach suggesting the possibility that this latter question might be equivalent to Invariant Subspace Problem for operators in Hilbert spaces.

In this section, we consider such approach in the setting of infinite dimensional separable complex Banach spaces $\X$. Let us consider the Banach space $\X \times \X$ endowed with the norm
$$\Vert \begin{bmatrix} x \\ y \end{bmatrix}\Vert=\Vert x \Vert +\Vert y \Vert,$$
for $x,y \in \X.$
Let $T \in L(\X)$ and define $A_T \in L(\X \times \X)$ by
$$
A_T\begin{bmatrix}
	x \\ y
\end{bmatrix}:=\begin{bmatrix}
y \\ Tx \end{bmatrix}=\begin{bmatrix}
0 & I \\ T & 0  \end{bmatrix} \begin{bmatrix}
x \\ y \end{bmatrix}.
$$
Foias, Jung, Ko and Pearcy considered the operator $A_T$, noting that $A_T^2$ has always a nontrivial closed invariant subspace. In this setting, the authors studied under which conditions $A_T$ has a nontrivial closed invariant subspace. The following result shows a necessary and sufficient conditions for any $T \in L(\X)$  to have a nontrivial closed invariant subspace.

\begin{thm} \label{teo1} Let $T \in L(\X).$ Then $T$ has a nontrivial closed invariant subspace if and only if $A_T$ and $A_{-T}$ has a common nontrivial closed invariant subspace.
\end{thm}

\begin{proof}
In order to prove the sufficiency, suppose that $T$ has a nontrivial closed invariant subspace $\St.$ Then $\M:=\St \times \St$ is a  common nontrivial closed invariant subspace for $A_T$ and $A_{-T}$. In fact, it is clear that $\M$ is closed and nontrivial. Also if
$\begin{bmatrix}
		s_1 \\ s_2
	\end{bmatrix} \in \M$ then
	$$A_T\begin{bmatrix}
		s_1 \\ s_2
	\end{bmatrix}=\begin{bmatrix}
		s_2 \\ Ts_1
	\end{bmatrix} \in \St \times \St \mbox{ and } A_{-T}\begin{bmatrix}
		s_1 \\ s_2
	\end{bmatrix}=\begin{bmatrix}
		s_2 \\ -Ts_1
	\end{bmatrix} \in \St \times \St.$$

For the necessity, without loss of generality, we may assume that $T$ is injective and has dense range.

Suppose that $\N \subseteq \X \times \X$ is a common nontrivial closed invariant subspace of $A_T$ and $A_{-T}$ and consider
$$\N_1:=\{x: \begin{bmatrix}
		x \\ 0 \end{bmatrix} \in \N\} \subseteq \X \mbox{ and }\N_2:=\{y:  \begin{bmatrix}
		0 \\ y \end{bmatrix} \in \N  \}\subseteq \X.$$
Both $\N_1$ and $\N_2$ are closed subspaces. Note that
$$A_T^2\begin{bmatrix}
	x \\ y \end{bmatrix}=\begin{bmatrix}
		Tx \\ Ty \end{bmatrix}$$
and $A_T^2(\N) \subseteq \N.$  Then $\N_1$ and $\N_2$ are both invariant subspaces under $T.$ Also, if $x \in \N_2$ then
$\begin{bmatrix}
		0 \\  x \end{bmatrix} \in \N$ and
$$A_T\begin{bmatrix}
		0 \\ x
	\end{bmatrix}=\begin{bmatrix}
		x\\ 0
	\end{bmatrix} \in \N
$$
so that $x \in \N_1.$  Hence $\N_2 \subseteq \N_1$ and, from $\N_1 \times \N_2 \subseteq \N$ it follows that $\N_2$ is not $\X.$
	
We only need to show that $\N_2$ is not $\{0\}.$ Since $\N\not =\{0\} \times \{0\}$ there exists $\begin{bmatrix}
	x \\ y
\end{bmatrix} \not = \begin{bmatrix}
0 \\ 0
\end{bmatrix}$ such that $\begin{bmatrix}
x \\ y
\end{bmatrix} \in \N.$
If $x=0$ then $y\not =0$ and $\begin{bmatrix}
	0 \\ y
\end{bmatrix} \in \N,$ so that $y \in \N_2$ and $\N_2 \not =\{0\}.$ If $x \not =0,$ then
$$A_T\begin{bmatrix}
		x \\ y
	\end{bmatrix}=\begin{bmatrix}
		y \\ Tx
	\end{bmatrix} \in \N, \mbox { and } A_{-T}\begin{bmatrix}
		x \\ y
	\end{bmatrix}=\begin{bmatrix}
		y \\ -Tx
	\end{bmatrix} \in \N.$$
	So that  $$\begin{bmatrix}
		0\\ 2Tx
	\end{bmatrix}= \begin{bmatrix}
		y\\ Tx
	\end{bmatrix}- \begin{bmatrix}
		y\\ -Tx
	\end{bmatrix} \in \N.$$
	
Then $Tx \in \N_2.$ Since $T$ is assumed to be injective, $Tx \not =0$ and $\N_2 \not =\{0\}.$  Hence  $\N_2$ is a nontrivial closed invariant subspace of $T$, which proves the theorem.
\end{proof}

As a consequence, we note the following regarding \emph{hyperinvariant subspaces}:

\begin{cor} Let $T \in L(\X).$ If $A_T^2$ has a nontrivial closed hyperinvariant subspace then $T$ has a nontrivial closed invariant subspace.
\end{cor}
\begin{proof} Clearly $A_T$ and $A_{-T}$ commute with $A_T^2.$ Then $A_T$ and $A_{-T}$ has a common nontrivial closed invariant subspace. Then, by  Theorem \ref{teo1}, $T$ has a nontrivial closed invariant subspace.
\end{proof}

\bigskip

From now on we consider bounded linear operators on a separable complex Hilbert space $\HH.$
One observation about $A_T$ and $A_{-T}$ in the Hilbert space setting is the following:

\begin{prop}
Assume that $T$ is an injective, dense range linear bounded operator acting on $\HH.$ Then $A_T$ and $iA_{-T}$ are quasisimilar operators.
\end{prop}

Recall that two operators $T,\, \tilde{T} \in L(H)$ are \textit{quasisimilar} if there exist operators $X,Y \in L(H)$ with $\ker X = \ker X^* = \ker Y = \ker Y^* = \{0\}$ satisfying $$ TX = X\tilde{T}, \quad YT = \tilde{T}Y.$$ Note that if $T$ and $\tilde{T}$ are quasisimilar and $T$ has a non trivial closed hyperinvariant subspace, so does $\tilde{T}$.

\begin{proof} Set $A:=A_T$ and $B:=iA_{-T}$. Then $A^2=B^2.$ Take $X:=A+B=\begin{bmatrix}
		0 & (1+i)I \\ (1-i)T & 0  \end{bmatrix}.$ Since $T$ is injective and has dense range, the operator $X$ is a quasiaffinity. In addition, $$XA=(A+B)A=A^2+BA=B^2+BA=B(A+B)=BX$$
and
$$AX=A(A+B)=A^2+AB=B^2+AB=(A+B)B=XB,$$
which yields the statement.
\end{proof}

In \cite{Nordgren}, the authors showed that every operator on $\HH$ has a non-trivial closed invariant subspace if and only if every pair of operators $\{A, B\}$ satisfying $p(A)=q(B)=0$ has a common nontrivial closed invariant subspace, with $p$ and $q$ any polynomial of degree $2.$ When $A$ and $B$ commute or anticommute (i.e. $AB=-BA$) the following hold.

\begin{prop} Every operator has a nontrivial closed invariant subspace if and only if every pair $\{A, B\}$ of commuting operators satisfying $A^2=B^2$ has a common nontrivial closed invariant subspace.
\end{prop}

\begin{proof}  Suppose that every operator has a nontrivial closed invariant subspace and let $\{A, B\}$ be a pair of commuting operators satisfying $A^2=B^2.$ Then $A$ (and $B$) has a nontrivial closed invariant subspace and, by Proposition \ref{Matache}, $A$ and $B$ has a common nontrivial closed invariant subspace.
Conversely, let $T \in L(\HH)$ then clearly $\{T,T\}$ is a pair of commuting operators satisfying the hypothesis. Hence $T$ has a nontrivial closed invariant subspace.
\end{proof}

\begin{prop} Every operator has a nontrivial closed invariant subspace if and only if every pair $\{A, B\}$ of anticommuting operators satisfying $A^2=-B^2$ has a common nontrivial closed invariant subspace.
\end{prop}

\begin{proof}  Suppose that every operator has a nontrivial closed invariant subspace and let  $\{A, B\}$ be a pair of anticommuting operators satisfying $A^2=-B^2.$  Then $(A-B)^2=A^2+B^2-AB-BA=0$ and $(A+B)^2=A^2+B^2+AB+BA=0.$ So that $(A-B)^2=(A+B)^2=0$ and, by \cite[Theorem 1]{Nordgren}, $A+B$ and $A-B$ have a common nontrivial closed invariant subspace. Then $A$ and $B$ have a common nontrivial closed invariant subspace.
	
Conversely, let $T \in L(\HH)$ and take $A=A_T$ and $B=A_{-T}$ then $A^2=-B^2$ and $AB=-BA.$ By hypothesis $A_T$ and $A_{-T}$ has a common nontrivial closed invariant subspace. Then, by Theorem \ref{teo1}, $T$ has a nontrivial closed invariant subspace.
\end{proof}

\section*{Acknowledgements}	
	
\noindent
The authors would like to thank the anonymous referee for a careful reading of the manuscript and for pointing out an oversight in the previous version of the Theorem \ref{thm1}.

	\affiliationone{Eva A. Gallardo-Guti\'errez\\
Dpto. de An\'alisis Matem\'atico y Matem\'atica Aplicada,\\
Facultad de Matem\'aticas, \\
Universidad Complutense de Madrid  and \\
Instituto de Ciencias Matem\'aticas ICMAT (CSIC-UAM-UC3M-UCM) \\
Madrid, Spain
\email{eva.gallardo@mat.ucm.es}}
\affiliationtwo{Maximiliano Contino \\
Dpto. de An\'alisis Matem\'atico y Matem\'atica Aplicada,\\
Facultad de Matem\'aticas, \\
Universidad Complutense de Madrid,\\
Madrid, Spain and \\
Instituto Argentino de Matemática \\
Alberto P. Calderón-CONICET \\
Buenos Aires, Argentina and \\
Dpto. de Matemática, Facultad de Ingeniería, \\
Universidad de Buenos Aires\\
	Buenos Aires, Argentina
\email{mcontino@fi.uba.ar}}

\end{document}